\documentclass[12pt]{amsart}
\usepackage{latexsym,amsmath,amsfonts,amsthm,amssymb}
\usepackage{version}

\begin{document}

\newcommand{\ds}{\displaystyle} 

\newcommand{\R}{{\mathbb R}}
\newcommand{\N}{{\mathbb N}}
\newcommand{\C}{{\mathbb C}}
\newcommand{\Q}{{\mathbb Q}}
\newcommand{\Z}{{\mathbb Z}}
\newcommand{\Rn}{{\mathbb R}^n}
\newcommand{\M}{{ \mathcal H}}

\newcommand{\1}{{{\mathbf 1}}}

\newcommand{\SX}{{ l^1S}}
\renewcommand{\H}{\dot{H}}
\newcommand{\tS}{\tilde{S}}
\newcommand{\tT}{\tilde{T}}
\newcommand{\tW}{\tilde{W}}
\newcommand{\tX}{\tilde{X}}
\newcommand{\tK}{\tilde{K}}

\newcommand{\tx}{\tilde{x}}
\newcommand{\txi}{\tilde{\xi}}

\newcommand{\tA}{\tilde{A}}
\newcommand{\tI}{{\tilde I}}
\newcommand{\tP}{\tilde{P}}

\newcommand{\la}{{\langle}}
\newcommand{\ra}{{\rangle}}

\newcommand{\supp}{{\text{supp }}}

\newcommand{\bc}{\begin{2}\begin{com}}
\newcommand{\ec}{\end{com}\end{2}}

\renewcommand{\varepsilon}{\epsilon}
\renewcommand{\div}{{\text{div}\,}}

\newcommand{\ext}{{\R^n\backslash\Omega}}
\newcommand{\bdy}{{\partial\Omega}}

\newtheorem{theorem}{Theorem}
\newtheorem{lemma}[theorem]{Lemma}
\newtheorem{proposition}[theorem]{Proposition}
\newtheorem{corollary}[theorem]{Corollary}
\newtheorem{definition}[theorem]{Definition}
\newtheorem{remark}[theorem]{Remark}
\newtheorem{assumption}[theorem]{Assumption}

\newcommand{\bb}{\beta}
\newcommand{\e}{\epsilon}
\newcommand{\oo}{\omega}
\newcommand{\dd}{\delta}
\renewcommand{\d}{{\partial}}
\renewcommand{\l}{\lambda}
\newcommand{\p}{{\partial}}

\newcommand{\ang}{{\not\negmedspace\nabla}}

\newenvironment{com}{\begin{quotation}{\leftmargin .25in\rightmargin .25in}\sffamily \footnotesize $\clubsuit$}
               {$\spadesuit$\end{quotation}\par\bigskip}
\newenvironment{com2}{\sffamily\footnotesize $\clubsuit$ }{ $\spadesuit$}

\bibliographystyle{plain}

\title{Decay estimates for variable coefficient wave equations
in exterior domains}

\author[J. Metcalfe]
{Jason Metcalfe}

\address{Department of Mathematics, University of North Carolina,
  Chapel Hill, NC  27599-3250, USA}
\email{metcalfe@email.unc.edu}

\author[D. Tataru]
{Daniel Tataru}

\address{Mathematics Department, University of California \\
Berkeley, CA 94720-3840, USA}
\email{tataru@math.berkeley.edu}

\thanks{ 
The work of the first author was supported in part by NSF grant DMS0800678.
The work of
the second author was supported in part by
NSF grants DMS0354539 and DMS0301122.}

\baselineskip 18pt

\begin{abstract}
  In this article we consider variable coefficient, time dependent
  wave equations in exterior domains $\R \times (\R^n \setminus
  \Omega)$, $n\geq 3$.  We prove localized energy estimates if
  $\Omega$ is star-shaped, and global in time Strichartz estimates if
  $\Omega$ is strictly convex.
\end{abstract}

\includeversion{jm}

\maketitle

\section{Introduction}

Our goal, in this article, is to prove analogs of the well known
Strichartz estimates and localized energy estimates for variable
coefficient wave equations in exterior domains.  We consider
long-range perturbations of the flat metric, and we take the obstacle
to be star-shaped.  The localized energy estimates are obtained under
a smallness assumption for the long range perturbation.
Global-in-time Strichartz estimates are then proved assuming the
local-in-time Strichartz estimates, which are known to hold for
strictly convex obstacles.

For the constant coefficient wave equation $\Box =
\partial_t^2-\Delta$ in $\R\times\R^n$, $n\ge 2$, we have that
solutions to the Cauchy problem
\begin{equation}
\Box u = f,\quad u(0)=u_0,\quad \partial_t u(0)=u_1,
\label{cccp}
\end{equation}
satisfy the Strichartz estimates\footnote{Here and throughout, we
  shall use $\nabla$ to denote a space-time gradient unless otherwise
  specified with subscripts.}
\[
\||D_x|^{-\rho_1} \nabla u\|_{L^{p_1}L^{q_1}}\lesssim \|\nabla
u(0)\|_{L^2} + \||D_x|^{\rho_2} \Box
u\|_{L^{p_2'}L^{q_2'}},
\]
for 
Strichartz admissible exponents $(\rho_1,p_1,q_1)$ and
$(\rho_2,p_2,q_2)$.  Here, exponents $(\rho,p,q)$ are called
Strichartz admissible if $2\le p,q\le \infty$,
\[ \rho = \frac{n}{2}-\frac{n}{q}-\frac{1}{p},\quad \frac{2}{p}\le \frac{n-1}{2}\Bigl(1-\frac{2}{q}\Bigr),\]
and $(\rho,p,q)\neq (1,2,\infty)$ when $n=3$.  

The Strichartz estimates follow via a $TT^*$ argument and the
Hardy-Littlewood-Sobolev inequality from the dispersive estimates,
\[
\| |D_x|^{-\frac{n+1}2 (1-\frac2q)} \nabla u(t)\|_{L^{q}} \lesssim
 t^{-\frac{n-1}2 (1-\frac2q)} \|u_1\|_{L^{q'}}, \qquad 2 \leq q < \infty
\]
for solutions to \eqref{cccp} with $u_0 = 0$, $f=0$. This in turn is obtained
by interpolating between a $L^2 \to L^2$ energy estimate and an
$L^1\to L^\infty$ dispersive bound which provides $O(t^{-(n-1)/2})$
type decay.  Estimates of this form originated in the work \cite{Str},
and as stated are the culmination of several subsequent works.  The
endpoint estimate $(p,q)=\Bigl(2,\frac{2(n-1)}{n-3}\Bigr)$ was most
recently obtained in \cite{MR1646048}, and we refer the interested
reader to the references therein for a more complete history.

The second estimate which shall be explored is the localized energy
estimate, a version of which states
\begin{equation}\label{kss}
  \sup_j \|\la x\ra^{-1/2} \nabla u\|_{L^2(\R\times \{|x|\in
    [2^{j-1},2^j]\})}
  \lesssim \|\nabla u(0)\|_{L^2} + \sum_k \|\la x\ra^{1/2} \Box
  u\|_{L^2(\R\times \{|x|\in [2^{k-1},2^k]\})}\end{equation}
in the constant coefficient case.  These estimates can be proved using
a positive commutator argument with a multiplier which is roughly of
the form $f(r)\partial_r$ when $n\ge 3$ and are quite akin to the
bounds found in, e.g., \cite{morawetz}, \cite{Strauss},
\cite{MR1308623}, \cite{smithsogge}, 
\cite{KSS}, and \cite{MR2128434}.  See also \cite{Alinhac},
\cite{MetSo}, \cite{globalw} for certain estimates for small 
perturbations of the
d'Alembertian.

Variants of these estimates for constant coefficient wave equations
are also known in exterior domains.  Here, $u$ is replaced by a
solution to
\[\Box u = F,\quad u|_{\partial \Omega}=0,\quad u(0)=u_0,\quad
\partial_t u(0)=u_1,\quad (t,x)\in \R\times \ext\] where $\Omega$ is a
bounded set with smooth boundary.  The localized energy estimates have
played a key role in proving a number of long time existence results
for nonlinear wave equations in exterior domains.  See, e.g.,
\cite{KSS} and \cite{MetSo2, MetSo} for their proof and application.
Here, it is convenient to assume that the obstacle $\Omega$ is
star-shaped, though certain estimates are known (see e.g.
\cite{MetSo2}, \cite{BurqGlobal}) in more general settings.  Exterior
to star-shaped obstacles, the estimates for small perturbations of
$\Box$ continue to hold (see \cite{MetSo2}).  This, however, only
works for $n\ge 3$, and the bound which results is not strong enough
in order to prove the Strichartz estimates which we desire.  As such,
we shall, in the sequel, couple this bound with certain frequency
localized versions of the estimate in order to prove the Strichartz
estimates.  For time independent perturbations, one may permit more
general geometries.  See, e.g., \cite{BurqGlobal}.

Certain global-in-time Strichartz estimates are also known in exterior
domains, but, except for certain very special cases (see \cite{BLP},
\cite{bss}, which are closely based on \cite{SmSo}), require that the
obstacle be strictly convex.  Local-in-time estimates were shown in
\cite{SmSoLocal} for convex obstacles, and using these estimates,
global estimates were constructed in \cite{smithsogge} for $n$ odd and
\cite{BurqGlobal} and \cite{MetGlobal} for general $n$.  See, also,
\cite{hmssz}.

In the present article, we explore variable coefficient cases of these
estimates.  Here, $\Box$ is replaced by the second order hyperbolic
operator
\[
P(t,x,D) = D_i a^{ij}(t,x) D_j+ b^i(t,x) D_i + c(t,x),
\]
where $D_0=D_t$ is understood.  We assume that $(a^{ij})$ has
signature $(n,1)$ and that $a^{00}<0$, i.e. that time slices are
space-like.  We shall then consider the initial value boundary value
problem
\begin{equation}
  \label{maineq}
Pu=f,\quad u|_{\partial\Omega}=0,\quad u(0)=u_0,\quad \partial_t
u(0)=u_1,\quad (t,x)\in \R\times \ext.
\end{equation}

When $\Omega=\emptyset$ and $b^i \equiv c\equiv 0$, the problem of
proving Strichatz estimates is understood locally, and of course,
localized energy estimates are trivial locally-in-time.  For smooth
coefficients, Strichartz estimates were first proved in
\cite{MR1168960} using Fourier integral operators.  Using a wave
packet decomposition, Strichartz estimates were obtained in
\cite{MR1644105} for $C^{1,1}$ coefficients in spatial dimensions
$n=2,3$.  Using instead an approach based on the FBI transform, these
estimates were extended to all dimensions in \cite{nlw, cs, lp}.  For
rougher coefficients, the Strichartz estimates as stated above are
lost (see \cite{SS}, \cite{MR1909638}) and only certain estimates with
losses are available \cite{cs, lp}.  When the boundary is nonempty,
far less is known, and we can only refer to the results of
\cite{SmSoLocal} for smooth time independent coefficients, $b^i\equiv
c\equiv 0$, and $\Omega$ strictly geodesically convex.  The proof of
these estimates is quite involved and uses a Melrose-Taylor parametrix
to approximate the reflected solution.

For the boundaryless problem, global-in-time localized energy
estimates and Strichartz estimates were recently shown in
\cite{globalw} for small, $C^2$, long-range perturbations.  The former
follow from a positive commutator argument with a multiplier which is
akin to what we present in the sequel.  For the latter, an outgoing
parametrix is constructed using a time-dependent FBI transform in a
fashion which is reminiscent to that of the preceding work \cite{gS}
on Schr\"odinger equations.  Upon conjugating the half-wave equation
by the FBI transform, one obtains a degenerate parabolic equation due
to a nontrivial second order term in the asymptotic expansion.  Here,
the bounds from \cite{gS}, which are based on the maximum principle,
may be cited.  The errors in this parametrix construction are small in
the localized energy spaces, which again are similar to those below,
and it is shown that the global Strichartz estimates follow from the
localized energy estimates.

The aim of the present article is to combine the approach of
\cite{globalw} with analogs of those from \cite{smithsogge},
\cite{BurqGlobal}, and \cite{MetGlobal} to show that global-in-time
Strichartz estimates in exterior domains follow from the localized
energy estimates and local-in-time Strichartz estimates for the
boundary value problem.  As we shall show the localized energy
estimates for small perturbations outside of star-shaped obstacles,
the global Strichartz estimates shall then follow for convex obstacles
from the estimates of \cite{SmSoLocal}.
 
Let us now more precisely describe our assumptions.  We shall look at
certain long range perturbations of Minkowski space.  To state this,
we set
\[D_0=\{|x|\le 2\},\quad D_j = \{2^j\le |x|\le 2^{j+1}\},\quad
j=1,2,\dots\]
and
\[
A_j = \R \times D_j, \qquad A_{<j} =  \R \times \{|x| \leq
2^{j}\}.
\]
We shall then assume that
\begin{equation}
  \label{coeff}
  \sum_{j \in \N} \sup_{A_j\cap (\R\times\ext)} \la x\ra^2 |\nabla^2 a(t,x)| + \la x\ra |\nabla a(t,x)|
  + |a(t,x)-I_n| \leq \e
\end{equation}
and, for the lower order terms, 
\begin{equation}
  \label{coeffb}
  \sum_{j \in \N} \sup_{A_j\cap (\R\times\ext)}  \la x\ra^2 |\nabla b(t,x)|
  + \la x\ra |b(t,x)| \leq \e
\end{equation}
\begin{equation}
  \label{coeffcc}
 \sum_{j \in \N} \sup_{A_j\cap (\R\times\ext)}  \la x\ra^2 |c(t,x)|
 \leq \e.
\end{equation}
If $\e$ is small enough then \eqref{coeff} precludes the existence of trapped
rays, while for arbitrary $\e$ it restricts the trapped rays to
finitely many dyadic regions. 

We now define the localized energy spaces that we shall use.  
We begin with an initial choice which is convenient for 
the local energy estimates but not so much for the Strichartz
estimates. Precisely, we define the localized energy space $LE_0$ as
 \[
\|\phi\|_{LE_{0}} = \sup_{j\ge 0}\Bigl(2^{-j/2} \|\nabla
 \phi\|_{L^2(A_j\cap (\R\times\ext))} + 2^{-3j/2}
 \|\phi\|_{L^2(A_j\cap(\R\times\ext))}\Bigr),
\]
while for the forcing term we set
 \[
\|f\|_{LE_0^*} = \sum_{k\ge 0} 2^{k/2} \|f\|_{L^2(A_k\cap
  (\R\times\ext))}.
\]
The local energy bounds in these spaces shall  follow from the 
arguments in \cite{MetSo}.

On the other hand, for the Strichartz estimates, we shall introduce
frequency localized spaces as in \cite{globalw}, as well as the
earlier work \cite{gS}.  We use a Littlewood-Paley decomposition in
frequency,
\[ 1= \sum_{k=-\infty}^\infty S_k(D), \quad \text{supp
}s_k(\xi)\subset \{2^{k-1}<|\xi|<2^{k+1}\}\]
and for each $k\in \Z$, we use
\[\|\phi\|_{X_k}=2^{-k^-/2} \|\phi\|_{L^2(A_{<k^-})} + \sup_{j\ge k^-}
\||x|^{-1/2}\phi\|_{L^2(A_j)}\]
to measure functions of frequency $2^k$.  Here $k^-=\frac{|k|-k}{2}$.  We then define the global
norm
\[
\|\phi\|^2_X=\sum_{k=-\infty}^\infty \|S_k \phi\|^2_{X_k}.
\]
Then for the local energy norm we use
\[
\|\phi\|^2_{LE_\infty} = \|\nabla \phi\|^2_X. 
\]
For the inhomogeneous term  we introduce the 
dual space $Y=X'$ with norm defined by
\[
\|f\|_Y^2=\sum_{k=-\infty}^\infty \|S_k f\|_{X_k'}^2.
\] 

To relate these spaces to the $LE_0$ respectively $LE_0^*$ 
we use Hardy type inequalities which are summarized 
in the following proposition:

\begin{proposition}\label{phardy}
  We have
  \begin{equation}
    \label{Hardy} \sup_j \||x|^{-1/2} u\|_{L^2(A_j)}\lesssim \|u\|_X
  \end{equation}
and
\begin{equation}\label{reverseHardy}
  \|u\|_Y\lesssim \sum_j \||x|^{1/2} u\|_{L^2(A_j)}.
\end{equation}
In addition, 
\begin{equation}\label{hardy.ge4}
\||x|^{-3/2} \phi\|_{L^2}\lesssim \|\nabla_x \phi\|_X,\quad n\ge
  4.\end{equation}
\end{proposition}
The first bound \eqref{Hardy} is a variant of a Hardy inequality, see
\cite[(16), Lemma 1]{globalw}, and also \cite{gS}. The second
\eqref{reverseHardy} is its dual. The bound \eqref{hardy.ge4}, proved
in \cite[Lemma 1]{globalw}, fails in dimension three. 


Now we turn our attention to the obstacle problem. For $R$
fixed so that $\Omega\subset\{|x|<R\}$, we select a smooth cutoff
$\chi$ with $\chi\equiv 1$ for $|x|<2R$ and $\supp\chi\subset\{|x|<4R\}$.
We shall use $\chi$ to partition the analysis into a portion near the
obstacle and a portion away from the obstacle.  In particular, we
define the localized energy space $LE \subset LE_0$ as
\[
\|\phi\|^2_{LE} = \| \phi\|^2_{LE_{0}} +
\|(1-\chi)\phi\|^2_{LE_\infty}.
\]
For the forcing term, we will respectively construct $LE^* \supset
LE_0^*$ by
\[
\|f\|^2_{LE^*} = \|\chi f\|^2_{LE_{0}^*} +
\|(1-\chi)f\|^2_{Y}, \qquad n \geq 4. 
\]
This choice is no longer appropriate in dimension $n=3$, as otherwise
the local $L^2$ control of the solution is lost. Instead we simply set
\[
\|f\|^2_{LE^*} = \| f\|^2_{LE_{0}^*}, \qquad n = 3.
\]

Using these space, we now define what it means for a solution to
satisfy our stronger localized energy estimates.
\begin{definition} 
  We say that the operator $P$ satisfies the localized energy
  estimates if for each initial data $(u_0,u_1) \in \dot H^{1} \times
  L^2$ and each inhomogeneous term $f \in LE^*$,
  there exists a unique solution $u$ to \eqref{maineq}
  with $u  \in LE$ 
which satisfies the bound
\begin{equation}
\|u\|_{LE} + \Bigl\|\frac{\partial u}{\partial
  \nu}\Bigr\|_{L^2(\partial \Omega)} \lesssim
\|\nabla u(0)\|_{L^2} + \|f\|_{LE^*}.
\end{equation}
\end{definition}

We prove that the localized energy estimates hold under the assumption
that $P$ is a small perturbation of the d'Alembertian:

\begin{theorem}\label{l3}
  Let $\Omega$ be a star-shaped domain.  Assume that the coefficients
  $a^{ij}$, $b^i$, and $c$ satisfy \eqref{coeff}, \eqref{coeffb}, and
  \eqref{coeffcc} with an $\e$ which is sufficiently small.  Then the
  operator $P$ satisfies the localized energy estimates
  globally-in-time for $n\ge 3$.
\end{theorem}

These results correspond to the $s=0$ results of \cite{globalw}.  Some
more general results are also available by permitting $s\neq 0$, but
for simplicity we shall not provide these details.

Once we have the local energy estimates, the next step is to 
prove the Strichartz estimates.  To do so, we shall assume that the
corresponding Strichartz estimate holds locally-in-time. 

\begin{definition}
  For a given operator $P$ and domain $\Omega$, we say that the
  local Strichartz estimate holds if
  \begin{equation}\label{locStr}
    \|\nabla u\|_{|D_x|^{\rho_1}L^{p_1}L^{q_1}([0,1]\times \ext)}
\lesssim \|\nabla u(0)\|_{L^2} + \|f\|_{|D_x|^{-\rho_2}L^{p'_2}L^{q_2'}([0,1]\times\ext)}
  \end{equation}
for any solution $u$ to \eqref{maineq}.
\end{definition}

As mentioned previously, \eqref{locStr} is only known under some
fairly restrictive hypotheses.  We show a conditional result which
says that the global-in-time Strichartz estimates follow from the
local-in-time estimates as well as the localized energy estimates.

\begin{theorem}\label{globalStrichartz}
  Let $\Omega$ be a domain such that $P$ satisfies both the localized
  energy estimates and the local Strichartz estimate.  Let $a^{ij},
  b^i, c$ satisfy \eqref{coeff}, \eqref{coeffb}, and \eqref{coeffcc}.  Let
  $(\rho_1,p_1,q_1)$ and $(\rho_2,p_2,q_2)$ be two Strichartz pairs.
  Then the solution $u$ to \eqref{maineq} satisfies
  \begin{equation}\label{glStr}
    \|\nabla u\|_{|D_x|^{\rho_1}L^{p_1}L^{q_1}}\lesssim \|\nabla
    u(0)\|_{L^2} + \|f\|_{|D_x|^{-\rho_2}L^{p_2'}L^{q_2'}}.
  \end{equation}
\end{theorem}

Notice that this conditional result does not require the $\e$ in
\eqref{coeff}, \eqref{coeffb}, and \eqref{coeffcc} to be small.  We
do, however, require this for our proof of the localized energy
estimates which are assumed in Theorem \ref{globalStrichartz}. 

As an example of an immediate corollary of the localized energy
estimates of Theorem \ref{l3} and the local Strichartz estimates of
\cite{SmSoLocal}, we have:

\begin{corollary}
  Let $n\ge 3$, and let $\Omega$ be a strictly convex domain.  Assume
  that the coefficients $a^{ij}$, $b^i$ and $c$ are time-independent
  in a neighborhood of $\Omega$ and satisfy \eqref{coeff},
  \eqref{coeffb} and \eqref{coeffcc} with an $\e$ which is sufficiently
  small.   Let
  $(\rho_1,p_1,q_1)$ and $(\rho_2,p_2,q_2)$ be two Strichartz pairs
  which satisfy
  \[ \frac{1}{p_1} =
  \Bigl(\frac{n-1}{2}\Bigr)\Bigl(\frac{1}{2}-\frac{1}{q_1}\Bigr),\quad
   \frac{1}{p_2'}=\Bigl(\frac{n-1}{2}\Bigr)\Bigl(\frac{1}{2}-\frac{1}{q_2'}\Bigr).\]
  Then the
  solution $u$ to \eqref{maineq} satisfies
\begin{equation}
\| \nabla u\|_{|D_x|^{\rho_1} L^{p_1}L^{q_1}}
\lesssim \|\nabla u(0)\|_{L^2} + 
\|f\|_{|D_x|^{-\rho_2} L^{p'_2}L^{q'_2}}.
\label{fse} \end{equation}
\label{tfse}\end{corollary}

This paper is organized as follows.  
In the next section,
we prove the localized energy estimates for small perturbations of the
d'Alembertian exterior to a star-shaped obstacle.  In the
last section, we prove Theorem \ref{globalStrichartz} which says that
global-in-time Strichartz estimates follow from the localized energy
estimates as well as the local Strichartz estimates.

\section{The localized energy estimates}
\label{mora}

In this section, we shall prove Theorem~\ref{l3}.

By combining the inclusions $LE \subset LE_0$, $LE_0^* \subset LE^*$ and
the bounds \eqref{hardy.ge4}, \eqref{coeffb}, and \eqref{coeffcc}, one
can easily prove the following which permits us to treat the lower
order terms perturbatively.  See, also, \cite[Lemma 3]{globalw}.

\begin{proposition}\label{bcbound}
  Let $b^i$, $c$ be as in \eqref{coeffb} and \eqref{coeffcc}
  respectively.  Then,
  \begin{align}
    \|b\nabla u\|_{LE^*} &\lesssim \e \|u\|_{LE}, \label{bbound}\\
    \|cu\|_{LE^*}&\lesssim \e\|u\|_{LE}.  \label{cbound}
  \end{align}
\end{proposition}

We now look at the proof of the localized energy estimates.  
Due to Proposition \ref{bcbound} we can assume that
$b=0, c=0$.  To prove the theorems, we use positive commutator
arguments.  We first do the analysis separately in the two regions.

\subsection{Analysis near $\Omega$ and classical Morawetz-type estimates}
Here we sketch the proof from \cite{MetSo} which gives an estimate
which is similar to \eqref{kss} for small perturbations of the d'Alembertian.
This estimate shall allow us to gain control of the solution near the
boundary.  It also permits local $L^2$ control of the solution, not
just the gradient in three dimensions.  The latter is necessary as the
required Hardy inequality which can be utilized in higher dimensions
corresponds to a false endpoint estimate in three dimensions.

The main estimate is the following:
\begin{proposition}\label{prop.kss}
  Let $\Omega$ be a star-shaped domain.  Assume that the coefficients
  $a^{ij}$, $b^i$, $c$ satisfy \eqref{coeff}, \eqref{coeffb}, and
  \eqref{coeffcc} respectively with an $\e$ which is sufficiently
  small.  Suppose that $\phi$ satisfies $P\phi=F$, $\phi|_\bdy=0$.  Then 
  \begin{equation}\label{main.kss}
    \|\phi\|_{LE_0} + \|\nabla \phi\|_{L^\infty L^2}
+ \|\partial_\nu \phi\|_{L^2(\bdy)}\lesssim \|\nabla \phi(0)\|_2 
+ \|F\|_{LE_0^*}.
  \end{equation}
\end{proposition}

\begin{proof}
  We provide only a terse proof.  The interested reader can refer to
  \cite{MetSo} for a more detailed proof.  For $f=\frac{r}{r+\rho}$,
  where $\rho$ is a fixed positive constant, we use a multiplier of
  the form
\[
\partial_t\phi + f(r)\partial_r\phi +\frac{n-1}{2}\frac{f(r)}{r}\phi.
\]
By multiplying $P\phi$ and integrating by parts, one obtains
\begin{equation}\label{commutator}
\begin{split}
  \!\! \int_0^T \!\! \int_\ext \frac{1}{2}f'(r)(\partial_r\phi)^2 & +
  \Bigl(\frac{f(r)}{r}-\frac{1}{2}f'(r)\Bigr)|\ang\phi|^2 +
  \frac{1}{2}f'(r)(\partial_t \phi)^2 -
  \frac{n-1}{4}\Delta\Bigl(\frac{f(r)}{r}\Bigr)\phi^2\:dxdt \!\!  \\ &
  -\frac{1}{2}\!\int_0^T\!\!\int_\bdy \!\!\frac{f(r)}{r}(\partial_\nu \phi)^2
  \la x,\nu\ra (a^{ij}\nu_i\nu_j)\:d\sigma dt + (1+O(\epsilon)) \|\nabla
  \phi(T)\|^2_2\\\lesssim & \ \|\nabla
  \phi(0)\|^2_2 + \int_0^T\int_\ext
  |F|\Bigl(|\partial_t\phi|+|f(r)\partial_r\phi|+
  \Bigl|\frac{f(r)}{r}\phi\Bigr|\Bigr)\:dx\:dt  \\ &\ +
  \int_0^T\int_\ext O\Bigl(\frac{|a-I|}{r}+|\nabla a|\Bigr)|\nabla
  \phi|\Bigl(|\nabla \phi| + \Bigl|\frac{\phi}{r}\Bigr|\Bigr)\:dx\:dt.
  \\\lesssim & \ \|\nabla \phi(0)\|^2_2 + \|F\|_{LE_0^*(0,T)}
  \|\phi\|_{LE_0(0,T)} + \epsilon \|\phi\|_{LE_0(0,T)}^2.
\end{split}
\end{equation}
Here, we have used the Hardy inequality $\||x|^{-1}\phi\|_2\lesssim
\|\nabla \phi\|_2$, $n\geq 3$, as well as \eqref{coeff}. 

All terms on the left are nonnegative.  By direct computation, the
first term controls
\[
\rho^{-1} \|\nabla \phi\|^2_{L^2([0,T]\times \{|x|\approx \rho\})}
+ \rho^{-3} \|\phi\|^2_{L^2([0,T]\times \{|x|\approx \rho\})}.
\]
Taking a supremum over dyadic $\rho$ provides a bound for the
$\|\phi\|_{LE_0(0,T)}$. In the second term
we have  $-\la x,\nu\ra \gtrsim 1$, which follows from the assumption
that $\Omega$ is star-shaped, and also  $a^{ij}\nu_i\nu_j\gtrsim 1$
which follows from \eqref{coeff}.  By simply taking
$\rho=1$, one can  bound the third term in the left of
\eqref{main.kss} by the right side of \eqref{commutator}. 
Thus we obtain
\[
   \|\phi\|_{LE_0(0,T)} + \|\nabla \phi(T)\|_{L^\infty L^2}
+ \|\partial_\nu \phi\|_{L^2(\bdy)}\lesssim \|\nabla \phi(0)\|^2_2 + \|F\|_{LE_0^*(0,T)}
  \|\phi\|_{LE_0(0,T)} + \epsilon \|\phi\|_{LE_0(0,T)}^2.
\]
The $LE_0$ terms on the right can be bootstrapped for $\epsilon$ small
which yields \eqref{main.kss}.
\end{proof}

\subsection{Analysis near $\infty$ and frequency localized estimates}

In this section, we briefly sketch the proof from \cite{globalw} for some
frequency localized versions of the localized energy estimates for the
boundaryless equation.  The main estimate here, which is from
\cite{globalw}, is the following.

\begin{proposition} \label{glw}
  Suppose that $a^{ij}$ are as in Theorem~\ref{l3} and $b=0$, $c=0$.
  Then for each initial data $(u_0,u_1)\in \dot{H}^1\times L^2$ and
  each inhomogeneous term $f\in Y\cap L^1L^2$, there exists a unique solution $u$
  to the boundaryless equation 
\[Pu=f,\quad u(0)=u_0,\quad \partial_t u(0)=u_1\]
satisfying
\begin{equation}\label{mtle}
\|\nabla u\|_{L^\infty L^2\cap X}\lesssim \|\nabla u(0)\|_{L^2} +
\|f\|_{L^1L^2 + Y}.\end{equation}
\end{proposition}

The proof here uses a multiplier of the form
\[D_t + \delta_0 Q + i\delta_1 B.\]
Here the parameters are chosen so that
\[\e\ll \delta_1\ll \delta\ll \delta_0\ll 1.\]
The multiplier $Q$ is given by
\[
Q = \sum_k S_k Q_k S_k
\]
where $Q_k$ are differential operators of the form 
\[
Q_k = (D_x x \phi_k( |x|) + \phi_k(|x|) xD_x).
\]
The $\phi_k$ are functions of the form
\[
\phi_k (x) =  2^{-k^-} \psi_k(2^{-k^-} \delta x)
\]
where for each $k$ the functions $\psi_k$ have the following
properties:
\begin{enumerate}
\item[(i)] $\psi_k(s) \approx (1+s)^{-1}$ for  $s > 0$ and $|\partial^j
\psi_k(s)| \lesssim (1+s)^{-j-1}$ for $j \leq 4$,

\item[(ii)] $ \psi_k(s) + s \psi_k'(s) \approx (1+s)^{-1} \alpha_k(s)$ 
 for  $s > 0$,

\item[(iii)] $\psi_k(|x|)$ is localized at frequency $ \ll 1$.
\end{enumerate}
The $\alpha_k$ are slowly varying functions that are related to the
bounds of the individual summands in \eqref{coeff}.  This construction
is reminiscent of those in \cite{gS}, \cite{MMT}, and \cite{globalw}.

For the Lagrangian term $B$, we fix a function $b$ satisfying
\[ b(s)\approx \frac{\alpha(s)}{1+s},\quad |b'(s)|\ll b(s).\]
Then, we set $B=\sum_k S_k 2^{-k^-} b(2^{-k^-}x)S_k$.

The computations, which are carried out in detail in \cite{globalw},
are akin to those outlined in the previous section.

\subsection{Proof of Theorem~\ref{l3}}
Consider first the three dimensional case. For $f \in LE^* = LE_0^*$
we can  use Proposition~\ref{prop.kss} to obtain
\[
  \|u\|_{LE_0} + \|\nabla u\|_{L^\infty L^2}
+ \|\partial_\nu u\|_{L^2(\bdy)}\lesssim \|\nabla u(0)\|_2 
+ \|f\|_{LE_0^*}.
\]
It remains to estimate $\| (1-\chi)u\|_{LE_\infty}$ with $\chi$ as
in  the definition of $LE$. By \eqref{mtle} we have
\[
\| (1-\chi)u\|_{LE_\infty} \lesssim \| \nabla
(1-\chi)u(0)\|_{L^2} + \| P [(1-\chi)u]  \|_{Y} \lesssim \|\nabla
u(0)\|_{L^2} + \| P   [(1-\chi)u]  \|_{LE^*_0}.
\]
Finally, to bound the last term we write
\[
P   [(1-\chi)u] = -[P,\chi]u +(1-\chi) f.
\]
The commutator has compact spatial support; therefore
\[
 \| P   [(1-\chi)u]  \|_{LE^*_0} \lesssim \|u\|_{LE_0} + \|f\|_{LE_0^*}
\]  
and the proof is concluded.

Consider now higher dimensions  $n \geq 4$.
For fixed $f\in LE^*$,  we first solve the boundaryless problem
\[
Pu_\infty = (1-\chi)f \in Y, \qquad u_\infty(0) = 0, \ \partial_t u_{\infty}(0) = 0
\]
using Proposition~\ref{glw}.  We consider $\chi_\infty$ which is
identically $1$ in a neighborhood of infinity and vanishes on $\supp
\chi$. For the function $\chi_\infty u_\infty$ we use the 
Hardy inequalities in Proposition~\ref{phardy} to write
\[
\|\chi_\infty u_\infty\|_{LE} \approx \|\nabla(\chi_\infty u_\infty)\|_{X}
\lesssim \| \nabla u_\infty\|_{X} \lesssim \|(1-\chi_\infty)f\|_{Y}.
\]
The remaining part $\psi = u - \psi_\infty u_\infty$ solves
\[
P \psi = \chi_\infty f + [P,\chi_\infty] u_\infty;
\]
therefore
\[
\|P \psi\|_{LE_0^*} \lesssim \|f\|_{LE^*} + \|u_\infty\|_{LE_0}
\lesssim \|f\|_{LE^*} + \|\nabla u_\infty\|_{X}
\lesssim \|f\|_{LE^*}.
\]
Then we estimate $\psi$ as in the three dimensional case. The 
proof is concluded.

\section{The Strichartz estimates}

In this final section, we prove Theorem \ref{globalStrichartz}, the
global Strichartz estimates.  We use fairly standard arguments to
accomplish this.  In a compact region about the obstacle, we prove the
global estimates using the local Strichartz estimates and the
localized energy estimates.  Near infinity, we use \cite{globalw}.
The two regions can then be glued together using the localized energy
estimates.

We shall utilize the following two propositions.  The first gives the
result when the forcing term is in the dual localized energy space.

\begin{proposition}\label{hom}
 Let $(\rho,p,q)$ be a Strichartz pair.
 Let $\Omega$ be a domain such that $P$ satisfies both the localized
 energy estimates and the homogeneous local Strichartz estimate with exponents $(\rho,p,q)$.
 Then for each $\phi\in LE$
 with $P \phi\in LE^*$, we have
 \begin{equation}\label{homest}
   \||D_x|^{-\rho}\nabla \phi\|^2_{L^pL^q}\lesssim \|\nabla
   \phi(0)\|_{L^2}^2 + \|\phi\|_{LE}^2 + \|P \phi\|^2_{LE^*}.
 \end{equation}
\end{proposition}

The second proposition allows us to gain control when the forcing term
is in a dual Strichartz space.

\begin{proposition}\label{nonhom}
  Let $(\rho_1,p_1,q_1)$ and $(\rho_2,p_2,q_2)$ be Strichartz pairs.
  Let $\Omega$ be a domain such that $P$ satisfies both the localized
  energy estimates and the local Strichartz estimate with exponents
  $(\rho_1,p_1,q_1)$, $(\rho_2,p_2,q_2)$.  Then there is a parametrix
  $K$ for $P$ with
  \begin{equation}\label{nonhomest}
    \|\nabla Kf\|^2_{L^\infty L^2} + \|Kf\|^2_{LE} + \||D_x|^{-\rho_1}\nabla
    Kf\|^2_{L^{p_1}L^{q_1}}\lesssim \||D_x|^{\rho_2}f\|^2_{L^{p_2'}L^{q_2'}}
  \end{equation}
  and
  \begin{equation}\label{errorest}
    \|P Kf - f\|_{LE^*}\lesssim \||D_x|^{\rho_2}f\|_{L^{p_2'}L^{q_2'}}.
  \end{equation}
\end{proposition}

We briefly delay the proofs and first apply the propositions to prove
Theorem \ref{globalStrichartz}.

\begin{proof}[Proof of Theorem \ref{globalStrichartz}]
For 
\[Pu=f+g,\quad f\in |D_x|^{-\rho_2}L^{p_2'}L^{q_2'},\, g\in LE^*,\]
we write
\[u=Kf+v.\]
The bound for $\nabla Kf$ follows immediately from \eqref{nonhomest}.

To bound $v$, we note that
\[P v = (1-PK)f+g.\]
Applying \eqref{homest} and the localized energy estimate,  we have
\[\||D_x|^{-\rho_1}\nabla v\|_{L^{p_1}L^{q_1}}\lesssim \|\nabla
u(0)\|_{L^2} + \|\nabla Kf\|_{L^\infty L^2} + \|(1-PK)f\|_{LE^*} + \|g\|_{LE^*}.\]
The Strichartz estimates \eqref{glStr} then
follow from \eqref{nonhomest} and \eqref{errorest}.  
\end{proof}

\begin{proof}[Proof of Proposition \ref{hom}]
We assume $P \phi\in Y$, and we write
\[\phi = \chi \phi + (1-\chi)\phi\]
with $\chi$ as in the definition of the $LE$ norm.
Since, using \eqref{reverseHardy}, the fundamental theorem of calculus, and
\eqref{Hardy}, we have
\[ \|[P, \chi] \phi\|_{LE^*}\lesssim \|\phi\|_{LE},\]
it suffices to show the estimate for $\phi_1=\chi \phi$,
$\phi_2=(1-\chi)\phi$ separately.

To show \eqref{homest} for $\phi_1$, we need only assume that $\phi_1$
and $P\phi_1$ are compactly supported, and we write
\[\phi_1 = \sum_{j\in \Z} \beta(t-j)\phi_1\]
for an appropriately chosen, smooth, compactly supported function
$\beta$.  By commuting $P$ and $\beta(t-j)$, we easily obtain
\[ \sum_{j\in \N} \|\beta(t-j)\phi_1\|^2_{LE} 
+ \|P
(\beta(t-j)\phi_1)\|^2_{L^1L^2}
\lesssim \|\phi_1\|^2_{LE}+ \|P \phi_1\|^2_{LE^*}.\]
Here, as above, we have also used \eqref{reverseHardy}, the
fundamental theorem of calculus, and \eqref{Hardy}.  Applying the
homogeneous local Strichartz estimate to each piece $\beta(t-j)\phi_1$
and using Duhamel's formula, the bound \eqref{homest} for $\phi_1$ follows
immediately from the square summability above.

On the other hand, $\phi_2$ solves a boundaryless equation, and the
estimate \eqref{homest} is just a restatement of \cite[Theorem 7]{globalw} with $s=0$.
This follows directly when $n\ge 4$ and easily from
\eqref{reverseHardy} when $n=3$.
\end{proof}

\begin{proof}[Proof of Proposition \ref{nonhom}]
  We split $f$ in a fashion similar to the above:
\[ f = \chi f + (1-\chi) f = f_1 + f_2.\]

For $f_1$, we write
\[f_1 = \sum_j \beta(t-j) f_1\]
where $\beta$ is supported in $[-1,1]$.  Let $\psi_j$ be the solution
to
\[P \psi_j = \beta(t-j)f_1.\]
By the local Strichartz estimate, we have
\[ \||D_x|^{-\rho_1} \nabla\psi_j\|_{L^{p_1}L^{q_1}(E_j)} + \|\nabla
\psi_j\|_{L^\infty L^2(E_j)} \lesssim \|\beta(t-j)|D_x|^{\rho_2}f_1\|_{L^{p_2'}L^{q_2'}}\]
where $E_j = [j-2,j+2]\times (\{|x|<2\}\cap \ext)$.  Letting
$\tilde{\beta}(t-j,r)$ be a cutoff which is supported in $E_j$ and is
identically one on the support of $\beta(t-j)\chi$, set $\phi_j = \tilde{\beta}(t-j,r)\psi_j$.
Then,
\begin{equation}\label{locnonhomest}
\||D_x|^{-\rho_1} \nabla \phi_j\|_{L^{p_1}L^{q_1}} + \|\nabla
\phi_j\|_{L^\infty L^2}\lesssim
\|\beta(t-j)|D_x|^{\rho_2} f_1\|_{L^{p_2'}L^{q_2'}}.\end{equation}
Moreover,
\[ P \phi_j - \beta(t-j)f_1 = [P,\tilde{\beta}(t-j,r)]\psi_j,\]
and thus,
\begin{equation}\label{locerror}
\|P \phi_j - \beta(t-j)f_1\|_{L^2}\lesssim \|\beta(t-j)|D_x|^{-\rho_2}f_1\|_{L^{p_2'}L^{q_2'}}.
\end{equation}
Setting 
\[Kf_1 = \sum_j \phi_j\]
and summing the bounds \eqref{locnonhomest} and \eqref{locerror}
yields the desired result for $f_1$.

For $f_2$, we solve the boundaryless equation
\[P \psi = f_2.\]
For a second cutoff $\tilde{\chi}$ which is 1 on the support of
$1-\chi$ and vanishes for $\{r<R\}$, we set
\[Kf_2=\tilde{\chi}\psi.\]
The following lemma, which is in essence from \cite[Theorem
  6]{globalw}, applied to $\psi$ then easily yields the desired bounds.

\begin{lemma}
  Let $f\in |D_x|^{-\rho_2}L^{p_2'}L^{q_2'}$.  Then the forward solution
  $\psi$ to the boundaryless equation $P\psi = f$ satisfies the bound
  \begin{equation}
    \|\nabla \psi\|^2_{L^\infty L^2}+\|\psi\|^2_{LE} + \||D_x|^{-\rho_1}\nabla
    \psi\|^2_{L^{p_1}L^{q_1}} \lesssim \||D_x|^{\rho_2}f\|^2_{L^{p_2'}L^{q_2'}}.
  \end{equation}
\end{lemma}

It remains to prove the lemma.  From \cite[Theorem 6]{globalw}, we
have that
\begin{equation}
  \label{backref}
\|\nabla \psi\|^2_X + \||D_x|^{-\rho_1}\nabla
\psi\|^2_{L^{p_1}L^{q_1}} \lesssim \||D_x|^{\rho_2}f\|^2_{L^{p_2'}L^{q_2'}}.
\end{equation}
By \eqref{Hardy} we have
\[
\sup_{j\ge 0} 2^{-j/2} \|\nabla
 \psi\|_{L^2(A_j)}
\lesssim \|\nabla \psi\|_X.
\]
It remains only to show the uniform bound
\begin{equation}
  \label{ll2}
2^{-\frac{3j}2}\|\psi\|_{L^2(A_j)}\lesssim \||D_x|^{\rho_2} f\|_{L^{p_2'}L^{q_2'}}
\end{equation}
when $n=3$.
Let $H(t,s)$ be the forward fundamental solution to $P$. Then 
\[
\psi(t) = \int_{-\infty}^t  H(t,s)f(s) \ ds.
\]
Therefore \eqref{ll2} can be rewritten as
\[
2^{-\frac{3j}2}\left\| \int_{-\infty}^t H(t,s)f(s) \ ds
\right\|_{L^2(A_j)}\lesssim \||D_x|^{\rho_2} f\|_{L^{p_2'}L^{q_2'}}.
\]
 Since $p_2'<2$ for
Strichartz pairs in $n=3$, by the Christ-Kiselev lemma
\cite{christkiselev} (see also \cite{smithsogge}) it suffices to show that 
\begin{equation}
  \label{ll2a}
2^{-\frac{3j}{2}}\Bigl\|\int_{-\infty}^\infty H(t,s)f(s)\:ds\Bigr\|_{L^2(A_j)}\lesssim
\||D_x|^{\rho_2}f\|_{L^{p_2'}L^{q_2'}}.
\end{equation}
The function
\[
\psi_1(t)  = \int_{-\infty}^\infty H(t,s)f(s)\:ds
\]
solves $P \psi_1 = 0$, and from \eqref{backref} we have
\[
\|\nabla \psi_1\|_{L^\infty L^2} \lesssim \||D_x|^{\rho_2}
f\|_{L^{p_2'}L^{q_2'}}.
\]
On the other hand, from \eqref{main.kss} with $P\psi_1 = 0$ and $\Omega=\emptyset$, we
obtain
\[
2^{-\frac{3j}2}\|\psi_1\|_{L^2(A_j)}
\lesssim \|\nabla \psi_1(0)\|^2_2.
\]
Hence \eqref{ll2a} follows, and the proof is concluded.

\end{proof}

\bibliography{nls}

\begin{thebibliography}{10}

\bibitem{Alinhac}
Serge Alinhac.
\newblock On the {M}orawetz--{K}eel-{S}mith-{S}ogge inequality for the wave
  equation on a curved background.
\newblock {\em Publ. Res. Inst. Math. Sci.}, 42(3):705--720, 2006.

\bibitem{bss}
Matthew Blair, Hart~F. Smith, and Christopher~D. Sogge.
\newblock Strichartz estimates for the wave equation on manifolds with
  boundary.
\newblock preprint.

\bibitem{BurqGlobal}
N.~Burq.
\newblock Global {S}trichartz estimates for nontrapping geometries: about an
  article by {H}. {F}.\ {S}mith and {C}. {D}.\ {S}ogge: ``{G}lobal {S}trichartz
  estimates for nontrapping perturbations of the {L}aplacian'' [{C}omm.
  {P}artial {D}ifferential {E}quation {\bf 25} (2000), no. 11-12 2171--2183;
  {MR}1789924 (2001j:35180)].
\newblock {\em Comm. Partial Differential Equations}, 28(9-10):1675--1683,
  2003.

\bibitem{BLP}
Nicolas Burq, Gilles Lebeau, and Fabrice Planchon.
\newblock Global existence for energy critical waves in 3-{D} domains.
\newblock {\em J. Amer. Math. Soc.}, 21(3):831--845, 2008.

\bibitem{christkiselev}
Michael Christ and Alexander Kiselev.
\newblock Maximal functions associated to filtrations.
\newblock {\em J. Funct. Anal.}, 179(2):409--425, 2001.

\bibitem{hmssz}
Kunio Hidano, Jason Metcalfe, Hart~F. Smith, Christopher~D. Sogge, and Yi~Zhou.
\newblock On abstract strichartz estimates and the strauss conjecture for
  nontrapping obstacles.
\newblock preprint.

\bibitem{KSS}
Markus Keel, Hart~F. Smith, and Christopher~D. Sogge.
\newblock Almost global existence for some semilinear wave equations.
\newblock {\em J. Anal. Math.}, 87:265--279, 2002.
\newblock Dedicated to the memory of Thomas H.\ Wolff.

\bibitem{MR1646048}
Markus Keel and Terence Tao.
\newblock Endpoint {S}trichartz estimates.
\newblock {\em Amer. J. Math.}, 120(5):955--980, 1998.

\bibitem{MR1308623}
Carlos~E. Kenig, Gustavo Ponce, and Luis Vega.
\newblock On the {Z}akharov and {Z}akharov-{S}chulman systems.
\newblock {\em J. Funct. Anal.}, 127(1):204--234, 1995.

\bibitem{MMT}
Jeremy Marzuola, Jason Metcalfe, and Daniel Tataru.
\newblock Strichartz estimates and local smoothing estimates for
  asympototically flat {S}chr\"odinger equations.
\newblock {\em J. Funct. Anal.}, 255(6):1497--1553, 2008.

\bibitem{MetSo2}
Jason Metcalfe and Christopher~D. Sogge.
\newblock Hyperbolic trapped rays and global existence of quasilinear wave
  equations.
\newblock {\em Invent. Math.}, 159(1):75--117, 2005.

\bibitem{MetSo}
Jason Metcalfe and Christopher~D. Sogge.
\newblock Long-time existence of quasilinear wave equations exterior to
  star-shaped obstacles via energy methods.
\newblock {\em SIAM J. Math. Anal.}, 38(1):188--209 (electronic), 2006.

\bibitem{globalw}
Jason Metcalfe and Daniel Tataru.
\newblock Global parametrices and dispersive estimates for variable coefficient
  wave equations.
\newblock preprint.

\bibitem{MetGlobal}
Jason~L. Metcalfe.
\newblock Global {S}trichartz estimates for solutions to the wave equation
  exterior to a convex obstacle.
\newblock {\em Trans. Amer. Math. Soc.}, 356(12):4839--4855 (electronic), 2004.

\bibitem{MR1168960}
Gerd Mockenhaupt, Andreas Seeger, and Christopher~D. Sogge.
\newblock Local smoothing of {F}ourier integral operators and
  {C}arleson-{S}j\"olin estimates.
\newblock {\em J. Amer. Math. Soc.}, 6(1):65--130, 1993.

\bibitem{morawetz}
Cathleen~S. Morawetz.
\newblock Time decay for the nonlinear {K}lein-{G}ordon equations.
\newblock {\em Proc. Roy. Soc. Ser. A}, 306:291--296, 1968.

\bibitem{MR1644105}
Hart~F. Smith.
\newblock A parametrix construction for wave equations with {$C\sp {1,1}$}
  coefficients.
\newblock {\em Ann. Inst. Fourier (Grenoble)}, 48(3):797--835, 1998.

\bibitem{SS}
Hart~F. Smith and Christopher~D. Sogge.
\newblock On {S}trichartz and eigenfunction estimates for low regularity
  metrics.
\newblock {\em Math. Res. Lett.}, 1(6):729--737, 1994.

\bibitem{SmSoLocal}
Hart~F. Smith and Christopher~D. Sogge.
\newblock On the critical semilinear wave equation outside convex obstacles.
\newblock {\em J. Amer. Math. Soc.}, 8(4):879--916, 1995.

\bibitem{smithsogge}
Hart~F. Smith and Christopher~D. Sogge.
\newblock Global {S}trichartz estimates for nontrapping perturbations of the
  {L}aplacian.
\newblock {\em Comm. Partial Differential Equations}, 25(11-12):2171--2183,
  2000.

\bibitem{SmSo}
Hart~F. Smith and Christopher~D. Sogge.
\newblock On the {$L\sp p$} norm of spectral clusters for compact manifolds
  with boundary.
\newblock {\em Acta Math.}, 198(1):107--153, 2007.

\bibitem{MR1909638}
Hart~F. Smith and Daniel Tataru.
\newblock Sharp counterexamples for {S}trichartz estimates for low regularity
  metrics.
\newblock {\em Math. Res. Lett.}, 9(2-3):199--204, 2002.

\bibitem{MR2128434}
Jacob Sterbenz.
\newblock Angular regularity and {S}trichartz estimates for the wave equation.
\newblock {\em Int. Math. Res. Not.}, (4):187--231, 2005.
\newblock With an appendix by Igor Rodnianski.

\bibitem{Strauss}
Walter~A. Strauss.
\newblock Dispersal of waves vanishing on the boundary of an exterior domain.
\newblock {\em Comm. Pure Appl. Math.}, 28:265--278, 1975.

\bibitem{Str}
Robert~S. Strichartz.
\newblock Restrictions of {F}ourier transforms to quadratic surfaces and decay
  of solutions of wave equations.
\newblock {\em Duke Math. J.}, 44(3):705--714, 1977.

\bibitem{nlw}
Daniel Tataru.
\newblock Strichartz estimates for operators with nonsmooth coefficients and
  the nonlinear wave equation.
\newblock {\em Amer. J. Math.}, 122(2):349--376, 2000.

\bibitem{cs}
Daniel Tataru.
\newblock Strichartz estimates for second order hyperbolic operators with
  nonsmooth coefficients. {I}{I}.
\newblock {\em Amer. J. Math.}, 123(3):385--423, 2001.

\bibitem{lp}
Daniel Tataru.
\newblock Strichartz estimates for second order hyperbolic operators with
  nonsmooth coefficients. {III}.
\newblock {\em J. Amer. Math. Soc.}, 15(2):419--442 (electronic), 2002.

\bibitem{gS}
Daniel Tataru.
\newblock Parametrices and dispersive estimates for {S}chr\"odinger operators
  with variable coefficients.
\newblock {\em Amer. J. Math.}, 130(3):571--634, 2008.

\end{thebibliography}

\end{document}